\pdfoutput=1
\RequirePackage{ifpdf}
\ifpdf 
\documentclass[pdftex]{sigma}
\else
\documentclass{sigma}
\fi

\usepackage{mathtools, bbm}
\usepackage{xspace}
\usepackage{xifthen}

\numberwithin{equation}{section}
\newtheorem{Theorem}{Theorem}[section]

\newtheorem{Lemma}[Theorem]{Lemma}

\theoremstyle{definition}
\newtheorem{defn}[Theorem]{Definition}
\newtheorem{Example}[Theorem]{Example}
\newtheorem{Remark}[Theorem]{Remark}

\newcommand*{\Z}{\mathbb Z}		
\newcommand*{\R}{\mathbb R}		
\newcommand*{\C}{\mathbb C}		
\newcommand*{\one}{\mathbbm 1}		
\newcommand*{\tensor}{\otimes}		
	
\DeclarePairedDelimiter{\scal}{\langle}{\rangle}	
\DeclarePairedDelimiter{\norm}{\lVert}{\rVert} 	
\DeclareMathOperator{\lin}{lin}		
\DeclareMathOperator{\Tr}{Tr}		
\DeclareMathOperator{\id}{id}		

\newcommand*{\alg}{\mathcal}		
\newcommand*{\hilb}{\mathfrak}		
\newcommand*{\aA}{\alg A}			
\newcommand*{\aB}{\alg B}			
\newcommand*{\aG}{\alg G}			
\newcommand*{\hH}{\hilb H}			
\newcommand*{\SU}{\mathrm{SU}}		
\DeclarePairedDelimiterXPP{\lprod}[2]
	{\,_{#1}}						
	{\langle}						
	{\rangle}						
	{}								
	{#2}							
\DeclarePairedDelimiterX{\rprod}[2]{\langle}{\rangle_{#1}}{#2}
\newcommand*{\End}{\mathcal L}
\newcommand*{\acts}{\,.\,}
\DeclarePairedDelimiterX{\ketbra}[2]{\lvert}{\rvert}{#1 \delimsize\rangle \delimsize\langle #2}
\DeclarePairedDelimiterX{\braket}[2]{\langle}{\rangle}{#1 \delimsize\vert #2}
\DeclarePairedDelimiterX{\matrixel}[3]{\langle}{\rangle}{#1 \delimsize\vert #2 \delimsize\vert #3}

\begin{document}

\newcommand{\arXivNumber}{1701.05895}

\renewcommand{\PaperNumber}{062}

\FirstPageHeading

\ArticleName{Part~III, Free Actions of Compact Quantum Groups\\ on $\boldsymbol{{\rm C}^*}$-Algebras}
\ShortArticleName{Part~III, Free Actions of Compact Quantum Groups on ${\rm C}^*$-Algebras}

\Author{Kay SCHWIEGER~$^\dag$ and Stefan WAGNER~$^\ddag$}
\AuthorNameForHeading{K.~Schwieger and S.~Wagner}
\Address{$^\dag$~Iteratec GmbH, Stuttgart, Germany}
\EmailD{\href{mailto:kay.schwieger@gmail.com}{kay.schwieger@gmail.com}}

\Address{$^\ddag$~Blekinge Tekniska H\"ogskola, Sweden}
\EmailD{\href{mailto:stefan.wagner@bth.se}{stefan.wagner@bth.se}}

\ArticleDates{Received April 05, 2017, in f\/inal form August 05, 2017; Published online August 09, 2017}

\Abstract{We study and classify free actions of compact quantum groups on unital ${\rm C}^*$-algebras in terms of generalized factor systems. Moreover, we use these factor systems to show that all f\/inite coverings of irrational rotation ${\rm C}^*$-algebras are cleft.}
\Classification{46L85; 37B05; 55R10; 16D70}

\section{Introduction}

Free actions of classical groups on ${\rm C}^*$-algebras were f\/irst introduced under the name \emph{saturated actions} by Rief\/fel \cite{Rieffel91} (see also \cite{Phi87,Phi09}) and equivalent characterizations where given by Ellwood~\cite{Ell00} and by Gottman, Lazar, and Peligrad \cite{GooLaPe94, Pel88} (see also \cite{BaCoHa15}). This class of actions does not admit degeneracies that may be present in general actions. For this reasons they are easier to understand and to classify. Indeed, for compact Abelian groups, free and ergodic actions, i.e., free actions with trivial f\/ixed point algebra, were completely classif\/ied by Olesen, Pedersen and Takesaki in \cite{OlPeTa80} and independently by Albeverio and H\o egh--Krohn in \cite{AlHo80}. This classif\/ication was generalized to compact non-Abelian groups by the remarkable work of Wassermann \cite{Wass89, Wass88a, Wass88b}. According to \cite{AlHo80,OlPeTa80,Wass88a}, for a compact group $G$ there is a 1-to-1 correspondence between free and ergodic actions of $G$ and unitary 2-cocycles of the dual group. An analogous result in the context of compact quantum groups has been obtained by Bichon, De Rijdt and Vaes~\cite{BiRiVa06}. Extending this classif\/ication beyond the ergodic case is however not straightforward because, even for a commutative f\/ixed point algebra, the action cannot necessarily be decomposed into a~bundle of ergodic actions.

The study of non-ergodic free actions is also motivated by their role as noncommutative principal bundles in noncommutative geometry. In fact, by a classical result, having a free action of a compact group $G$ on a locally compact space $P$ is equivalent saying that $P$ carries the structure of a principal bundle over the quotient $X := P/G$ with structure group~$G$. Moreover, Rief\/fel showed that there is a 1-to-1 correspondence between classical free actions of compact groups on locally compact spaces and free actions of compact groups on commutative ${\rm C}^*$-algebras (cf.\ \cite[Proposition~7.1.12 and Theorem~7.2.6]{Phi87}). From this perspective, the notion of a free action on a ${\rm C}^*$-algebra provides a natural framework for noncommutative principal bundles, which become increasingly prevalent in application to geometry and physics. Regarding classif\/ication, the case of locally trivial principal bundles, that is, if $P$ is glued together from spaces of the form $U \times G$ with an open subset $U \subseteq X$, is very well-understood. This gluing immediately leads to $G$-valued cocycles. The corresponding cohomology theory, called \v{C}ech cohomology, gives a complete classif\/ication of locally trivial principal bundles with base space $X$ and structure group~$G$.

The present paper is a sequel of \cite{SchWa15} and \cite{SchWa16}, where we studied free actions of compact Abelian groups and so-called cleft actions, respectively. To be more precise, we achieved in~\cite{SchWa15} a complete classif\/ication of free, but not necessary ergodic actions of compact Abelian groups on unital ${\rm C}^*$-algebras. This classif\/ication extends the results of~\cite{AlHo80,OlPeTa80} and relies on the fact that the corresponding isotypic components are Morita self-equivalence over the f\/ixed point algebra. Moreover, we provided a classif\/ication of principal bundles with compact Abelian structure group which are not locally trivial. For free actions of non-Abelian compact groups the bimodule structure of the corresponding isotypic components is more subtle. For this reason we concentrated in~\cite{SchWa16} on a simple class of free actions of non-Abelian compact groups, namely cleft actions. Regarded as noncommutative principal bundles, these actions are characterized by the fact that all associated noncommutative vector bundles are trivial. In the present article we turn to the general case of free actions of compact quantum groups. The main objective of this article is to provide a complete description of free actions of compact quantum groups on unital ${\rm C}^*$-algebras in terms of so-called factor systems. Besides an interesting characterization of freeness, our approach uses the fact that nonergodic actions of compact quantum groups can be described in terms of weak unitary tensor functors, i.e., functors from the representation category of the underlying compact quantum group into the category of ${\rm C}^*$-correspondences over the corresponding f\/ixed point algebra (cf.~\cite[Section~2]{Ne13}). More detailedly, the paper is organized as follows.

After some preliminaries, we introduce in Section~\ref{section3} the notion of freeness for compact ${\rm C}^*$-dynamical systems and prove its equivalence to the Ellwood condition (Theorem~\ref{thm:equcondsatact}). We also list a few examples and establish the basis for our later classif\/ication in terms of generalized factor systems. In Section~\ref{section4} we show that every free compact ${\rm C}^*$-dynamical system gives rise to a so-called factor system and that free compact ${\rm C}^*$-dynamical systems can be classif\/ied up to equivalence by their associated factor system (Theorem~\ref{thm:equivalence}). This extends the results presented in part~2 of this series~\cite{SchWa16}, which deals with the particular class of cleft actions. Moreover, we give a characterization of cleft actions in terms of their factor systems. The purpose of Section~\ref{section5} is to show that the information provided by a factor system is enough to explicitly reconstruct the ${\rm C}^*$-dynamical system by adapting results of \cite{Ne13}. This completes our classif\/ication result showing that there is a 1-to-1 correspondence between free compact ${\rm C}^*$-dynamical systems and factor systems up to equivalence and conjugacy, respectively (Theorem~\ref{thm:main_class_thm}). As an application, we show in Section~\ref{section6} that f\/inite coverings of generic irrational rotation ${\rm C}^*$-algebras are always cleft (Theorem~\ref{thm:cleft covering}).

\section{Preliminaries and notations}\label{section2}

Our study is concerned with free actions of compact groups on unital ${\rm C}^*$-algebras and their classif\/ication in terms of generalized factor systems. Consequently, we use and blend tools from operator algebras and representation theory. In this preliminary section we provide def\/initions and notations which are repeatedly used in this article.

\subsection*{$\boldsymbol{{\rm C}^*}$-algebras}

Let $\aA$ be a unital ${\rm C}^*$-algebra. For the unit of $\aA$ we write $\one_\aA$ or simply $\one$. We will frequently deal with partial isometries, i.e., elements $v \in \aA$ such that $v^* v$ and $vv^*$ are projections. In this case $v^* v$ is called the cokernel projection and $vv^*$ the range projection. Moreover, we say that a projection $p$ is larger than the range of an element~$x$ if $px = x$, and we say that $p$ is larger than the cokernel of~$x$ if $xp = x$. All tensor products of ${\rm C}^*$-algebras are taken with respect to the minimal tensor product. We will frequently deal with multiple tensor products of unital ${\rm C}^*$-algebras $\aA$, $\aB$, and $\alg C$. If there is no ambiguity, we regard $\aA$, $\aB$, and $\alg C$ as subalgebras of $\aA \tensor \aB \tensor \alg C$ and extend maps on $\aA$, $\aB$, or $\alg C$ canonically by tensoring with the identity map. For sake of clarity we may occasionally use the leg numbering notation, e.g., for $x \in \aA \tensor \alg C$ we write~$x_{13}$ to denote the corresponding element in $\aA \tensor \aB \tensor \alg C$.

Inner products $\scal{\cdot, \cdot}$ on a Hilbert space is always assumed to be linear in the second component. For a Hilbert space $\hH_1, \hH_2$ we denote by $\End(\hH_1, \hH_2)$ the set of bounded linear operators \smash{$T\colon \hH_1 \to \hH_2$}. If $\hH_1 = \hH_2$ we brief\/ly write $\End(\hH_1)$. We use the Dirac notation to specify operators, i.e., for two vectors $v_1 \in \hH_1$, $v_2 \in \hH_2$ we write $\ketbra{v_2}{v_1}$ for the operator $v \mapsto \scal{v_1,v} v_2$.

\subsection*{Hilbert modules}

For a unital ${\rm C}^*$-algebra $\aA$ a right pre-Hilbert $\aA$-module is a right $\aA$-module $\hilb H$ equipped with a~sesquilinear map $\rprod{\aA}{\cdot, \cdot}\colon \hilb H \times \hilb H \to \aA$ that satisf\/ies the usual axioms of a def\/inite inner product with $\aA$-linearity in the second component. We call $\hH$ a right Hilbert $\aA$-module if $\hH$ is complete with respect to the norm $\norm{x}_{\hilb H} := \norm{\rprod{\aA}{x,x}}^{1/2}$. The right Hilbert \mbox{$\aA$-module} is called full if the two-sided ideal $\rprod{\aA}{\hH, \hH} := \overline{\lin \{ \rprod{\aA}{x,y} \,|\, x,y \in \hH\}}$ is dense in $\aA$. Since every dense ideal of~$\aA$ meets the invertible elements, in this case we have $\rprod{\aA}{\hH, \hH} = \aA$. Left (pre-) Hilbert $\aA$-modules are def\/ined in a similar way.

A~correspondence over $\aA$, or a right Hilbert $\aA$-bimodule, is a $\aA$-bimodule $\hilb H$ equipped with a $\aA$-valued inner product $\scal{\cdot, \cdot}_\aA$ which turns it into a right Hilbert $\aA$-module such that the left action of $\aA$ on $\hilb H$ is via adjointable operators. For two correspondences $\hilb H$ and $\hilb K$ over $\aA$ we denote by $\hilb H \tensor_\aA \hilb K$ their tensor product, on which the inner product is given by
	$\scal{x_1 \tensor y_1, x_2 \tensor y_2}_\aA = \scal{y_1, \scal{x_1,x_2}_\aA \,. \, y_2}_\aA$
for all $x_1, x_2 \in \hilb H$ and $y_1, y_2 \in \hilb K$.

\subsection*{Compact quantum groups}

We rely on the ${\rm C}^*$-algebraic notion of compact quantum groups as introduced by Worono\-wicz~\cite{Wor87a}. For an introduction and further details we recommend \cite{Commer16,NeTu13,Timm08}. A compact quantum group is given by a unital ${\rm C}^*$-algebra $\aG$ together with a (usually implicit) faithful, unital \mbox{$^*$-homo\-morphism} $\Delta\colon \aG \to \aG \tensor \aG$ satisfying the identity $(\Delta \tensor \id) \circ \Delta = (\id \tensor \Delta) \circ \Delta$ and such that $\Delta(\aG) (\one \tensor \aG)$ is dense in $\aG \tensor \aG$. It can be shown that there is a unique state $h\colon \aG \to \C$ such that $(\id \tensor h) \circ \Delta = h = (h \tensor \id) \circ \Delta$ (see~\cite{Wor87a}). This state is called the \emph{Haar state} of $\aG$. It is not faithful in general but via the GNS-construction we may replace $\aG$ by its reduced version on which the Haar state is faithful. Since $\aG$ and its reduce version behave identically with respect to their representation theory and their actions (see \cite[Section~4]{Commer16}), we will throughout the text assume that the Haar state on $\aG$ is faithful.

A unitary \emph{representation} of a compact quantum group $\aG$ on a f\/inite-dimensional Hilbert space~$V$ is a unitary element $\pi \in \End(V) \tensor \aG$ such that $\id \tensor \Delta(\pi) = \pi_{12} \pi_{13}$ in $\End(V) \tensor \aG \tensor \aG$. Unless explicitly stated otherwise, all representations are assumed unitary and f\/inite-dimensional. We recall that the set of equivalence classes of irreducible representations~$\hat \aG$ is countable and that the matrix coef\/f\/icients of all $\pi \in \hat \aG$ generate a dense $^*$-subalgebra of $\hat \aG$. Since all constructions behave naturally with respect to intertwiners we will not distinguish between a representation and its equivalence class. The tensor product of two representations $(\pi,V)$ and $(\rho, W)$ of~$\aG$ is the representation $(\pi \tensor \rho, V \tensor W)$ given by the unitary element $\pi \tensor \rho := \pi_{13} \rho_{23}$ in $\End(V) \tensor \End(W) \tensor \aG$. We also recall that for a representation $(\pi, V)$ of $\aG$ the contragradient representation is in general not unitary. Its normalization $(\bar\pi, \bar V)$ is called the \emph{conjugated representation}.

Some care has to be taken in the case that the Haar state is not tracial. Then the matrix coef\/f\/icients with respect to some chosen basis of $V$ are not orthogonal in general. However, if $\pi$ is irreducible, there is a unique positive, invertible operator $Q(\pi) \in \End(V)$ normalized to $\Tr[ Q(\pi)] = \Tr [ Q(\pi)^{-1} ]$ with
\begin{gather*}	\label{eq:*}
	\frac{\Tr[ Q(\pi) T]}{\Tr[ Q(\pi) ]}  \one_V= \id \tensor h \bigl( \pi (T\tensor \one_\aG) \pi^* \bigr),\qquad
	\frac{\Tr[ Q(\pi)^{-1} T ]}{\Tr[ Q(\pi) ]}  \one_V= \id \tensor h \bigl( \pi^* (T \tensor \one_\aG) \pi \bigr)
\end{gather*}
for every $T \in \End(V)$. The number $d_\pi := \Tr[Q(\pi)]$ is called the \emph{quantum dimension} of~$\pi$. The quantum dimension behaves nicely with respect to taking direct sums, tensor products, and conjugated representations.
An important detail for us is the fact that we may f\/ix intertwiners $R\colon \C \to V \tensor \bar V$ and $\bar R\colon \C \to \bar V \tensor V$ for all irreducible representation such that \mbox{$(R^* \tensor \id_V) (\id_V \tensor \bar R) = \id_V$}. In terms of an orthonormal basis $e_1, \dots, e_n \in V$ and its respective conjugated basis $\bar e_1, \dots, \bar e_n \in \bar V$ we typically choose
\begin{gather*}
R(1) = \sum\limits_{i=1}^n Q(\pi)^{1/2} e_i \tensor \bar e_i.
\end{gather*}

\subsection*{Actions of compact quantum groups}

An \emph{action} of a compact quantum group $\aG$ on a unital ${\rm C}^*$-algebra $\aA$ is a faithful, unital $^*$-homo\-morphism $\alpha\colon \aA \to \aA \tensor \aG$ that satisf\/ies $(\id \tensor \Delta) \circ \alpha = (\alpha \tensor \id) \circ \alpha$ and such that $(\one \tensor \aG)\, \alpha(\aA)$ is dense in~$\aA \tensor \aG$. Since we assume that the Haar state is faithful, the map $P_1 := (\id \tensor h) \circ \alpha$ is a~faithful conditional expectation onto the f\/ixed point algebra
\begin{align*}
\aA^\aG := \{x \in \aA \,|\, \alpha(x) = x \tensor \one_\aG \}.
\end{align*}
In particular, $\aA$ turns into a right pre-Hilbert $\aA^\aG$-bimodule with the $\aA^\aG$-valued inner product \mbox{$\rprod{\aA^\aG}{x,y} := P_1(x^* y)$} for $x,y \in \aA$. For each irreducible representation $\pi \in \hat G$ the projection $P_\pi\colon \aA \to \aA$ onto the \emph{$\pi$-isotypic component} $A(\pi) := P_\pi(\aA)$ is given by
\begin{gather*}
	P_\pi(a) := d_\pi  \Tr \tensor \id_\aA \tensor h \bigl( \bar\pi_{13} \, \alpha(a)_{23} \, Q(\bar\pi)^{-1}_1 \bigr), \qquad
	a \in \aA,
\end{gather*}
where the leg numbering refers to $\End(\bar V) \tensor \aA \tensor \aG$ (see \cite[Theorem~1.5]{Podles95}). The set $A(\pi)$ is in fact closed with respect to the inner product (see \cite[Corollary~2.6]{CoYa13a}) and hence a correspondence over~$\aA^\aG$. Furthermore, isotypic components for dif\/ferent $\pi \in \hat \aG$ are orthogonal with respect to the inner product and the sum $\sum\limits_{\pi \in \hat \aG} A(\pi)$ is dense in~$\aA$.

\section[Free ${\rm C}^*$-dynamical systems]{Free $\boldsymbol{{\rm C}^*}$-dynamical systems}\label{section3}

Throughout the presentation we discuss compact ${\rm C}^*$-dynamical systems $(\aA, \aG, \alpha)$, by which we mean a unital ${\rm C}^*$-algebra $\alg A$, a compact quantum group $\aG$, and an action $\alpha\colon \aA \to \aA \tensor \aG$. Given such a system, we recall that $\aA$ can be decomposed in terms of its isotypic compo\-nents~$A(\pi)$, $\pi \in \hat \aG$, and that each $A(\pi)$ is a correspondence over the f\/ixed point algebra $\aA^\aG$. For each irreducible representation $(\pi,V) \in \hat \aG$ we denote by $\Gamma(V)$ the multiplicity space of the conjugated representation~$\bar\pi$, which can be written in the form
\begin{gather*}
	\Gamma(V) = \{ x \in V \tensor \aA \,|\, \pi_{13} \id_V \tensor \alpha(x) = x \tensor \one_\aG \}.
\end{gather*}
This space is naturally a correspondence over $\aA^\aG$ with respect to the usual left and right multiplication and the restriction of the inner product $\rprod{\aA^\aG}{v \tensor a, w \tensor b} := \scal{v,w} \, a^*b$ for all $v,w \in V$ and $a, b \in \aA$. The $\pi$-isotypic component $A(\pi)$ is then as a correspondence isomorphic to $V \tensor \Gamma(\bar V)$ via the map $\varphi_\pi\colon V \tensor \Gamma(\bar V) \to A(\pi)$, $\varphi_\pi := \smash{d_\pi^{-1/2}} R^* \tensor \id_\aA$.

The mapping $(\pi, V) \mapsto \Gamma(V)$ can be extended to an additive functor from the representation category of $\aG$ into the category of ${\rm C}^*$- correspondences over~$\aA^\aG$. Since $\aA$ is the closure of the direct sum of its isotypic components and every isotypic component $A(\pi)$ is isomorphic to $V \tensor \Gamma( \bar V)$, this functor allows us to reconstruct the Hilbert $\aA^\aG$-bimodule structure of $\aA$ and the action~$\alpha$ up to a suitable closure. To recover the multiplication on $\aA$ we may look at the family of maps
\begin{gather}\label{eq:m(pi,rho)}
	m_{\pi, \rho}\colon  \ \Gamma(V) \tensor_\aB \Gamma(W) \to \Gamma(V \tensor W), \qquad 	m_{\pi, \rho}(x \tensor y) := x_{13} y_{23},
\end{gather}
for representations $(\pi,V)$, $(\rho,W)$ of $\aG$. Here the subindices on the right hand side refer to the leg numbering in $V \tensor W \tensor \aA$, that is, for elementary tensors $x = (v \tensor a)$ and $y = (w \tensor b)$ we write $x_{13} \, y_{23} = v \tensor w \tensor a b$ ($v \in V$, $w \in W$, $a,b \in \aA$). The functor $\Gamma\colon V \mapsto \Gamma(V)$ and the transformations $(m_{\pi,\rho})_{\pi,\rho}$ constitute a so-called weak tensor functor and allow to recover the reduced form of the compact ${\rm C}^*$-dynamical system $(\aA, \aG, \alpha)$ up to isomorphisms (see \cite[Section~2]{Ne13}).

To obtain a more concrete representation we restrict ourselves to the class of free action in the following sense. In addition to the above correspondence structure, we equip each multiplicity space $\Gamma(V)$ with the left $\End(V) \tensor \aA$-valued inner product given by
\begin{gather*}
\lprod{\End(V) \tensor \aA}{v \tensor a, w \tensor b} := \ketbra{v}{w} \tensor ab^*
\end{gather*}
for $v,w \in V$ and $a,b \in \aA$. A few moments thought show that this left inner product takes values in the ${\rm C}^*$-algebra $\{ x \in \End(V) \tensor \aA \,|\, \alpha(x) = \pi(x \tensor \one_\aG) \pi^* \}$ and that the only missing feature for $\Gamma(V)$ to be a Morita equivalence bimodule is that in general the left inner product need not be full. This requirement is what we demand for a free action:

\begin{defn} 	\label{def:freeness}
	A compact ${\rm C}^*$-dynamical system $(\aA, \aG, \alpha)$ is called \emph{free} if for every $(\pi,V) \in \hat \aG$ we have $\one \in \lprod{\End(V) \tensor \aA}{\Gamma(V), \Gamma(V)}$.
\end{defn}

There are various non-equivalent notions of freeness in the literature (see, e.g., \cite{EchNeOy09,Phi09} and references therein). The one given here was introduced for actions of classical compact groups by Rief\/fel \cite{Rieffel91} under the term saturated actions (see also \cite[Corollary~3.5]{Pel88} and \cite[Lemma~3.1]{GooLaPe94}) and already used in the other parts of this series \cite{SchWa15,SchWa16}, where some equivalent conditions are summarized. A seemingly dif\/ferent version of freeness for actions of compact quantum groups was recently exploited by De~Commer et al.\ \cite{BaCoHa15,CoYa13a} and is due to D.A.~Ellwood \cite{Ell00}. We recall that a~compact ${\rm C}^*$-dynamical system $(\aA, \aG, \alpha)$ is said to satisfy the \emph{Ellwood condition} if $(\aA \otimes \one)\alpha(\aA)$ is dense in~$\aA \tensor \aG$. For convenience we now summarize the equivalent conditions of freeness and provide proper references for the implications.

\begin{Theorem} \label{thm:equcondsatact}
Let $(\aA,\aG,\alpha)$ be a compact ${\rm C}^*$-dynamical system. Then the following conditions are equivalent:
	\begin{itemize}\itemsep=0pt
\item[$(a)$] The ${\rm C}^*$-dynamical system $(\aA,\aG,\alpha)$ is free.
\item[$(b)$] For all representations $(\pi,V)$, $(\rho,W)$ of $\aG$ the map $m_{\pi, \rho}$ defined in equation \eqref{eq:m(pi,rho)}
		has dense range or, equivalently, is surjective.
\item[$(c)$] The ${\rm C}^*$-dynamical system $(\aA,G,\alpha)$ satisfies the Ellwood condition.
	\end{itemize}
\end{Theorem}

The equivalence between (b) and (c) was proved quite recently in \cite[Theorem~0.4]{BaCoHa15}. For the implication (c) $\Rightarrow$ (a) we refer to the proof of \cite[Corollary~5.6]{CoYa13a}. Finally, the implication \mbox{(a) $\Rightarrow$ (b)} will follow immediately from the independent later results of Section~\ref{section4} and from Lemma~\ref{lem:unitary tensor functor}. Alternatively, (b) follows from (a) by observing that an equivalent way to formulate the condition in Def\/inition~\ref{def:freeness} is by saying that, for every representation $(\pi,V)$ of $\aG$, the right Hilbert $\aA$-module~$V \otimes \aA$ has a basis (in the sense of Hilbert modules) consisting of invariant elements.

We continue with a reformulation of freeness which will be convenient for our description of free ${\rm C}^*$-dynamical systems in terms of generalized factor systems.

\begin{Lemma}	\label{lem:freeness}
	A compact ${\rm C}^*$-dynamical system $(\aA, \aG, \alpha)$ is free if and only if for every representation $(\pi, V)$ of $\aG$ there is a finite-dimensional Hilbert space $\hH$ and a coisometry $s \in \End(\hH, V) \tensor \aA$ with $\pi_{13} \id \tensor \alpha(s)=s \tensor \one_\aG$.
\end{Lemma}
\begin{proof}For the ``if''-implication let $(\pi, V) \in \hat \aG$ and let $s \in \End(\hH,V) \tensor \aA$ be a coisometry with $\pi \alpha(s) = s$. Moreover, f\/ix an orthonormal basis of $\hH$ and denote by $s_k \in V \tensor \aA$ the columns of~$s$. Then
\begin{gather*}
\sum_{k=1}^n \lprod{\End(V) \tensor \aA}{s_k,s_k} = s s^* = \one.
\end{gather*}
For the converse implication, f\/irst observe that freeness of the ${\rm C}^*$-dynamical system $(\aA, \aG, \alpha)$ implies that, for each representation $(\pi,V)$ of~$\aG$, the space $\Gamma(V)$ is a Morita equivalence bimodule between the ${\rm C}^*$-algebras
$C(\pi) := \{x \in \End(V) \tensor \aA \,|\, \id \tensor \alpha(x) = \pi_{13} (x \tensor \one_\aG) \pi_{13}^* \}$  and~$\aA^\aG$.
Since~$C(\pi)$ is unital, there are elements $s_1, \dots, s_n \in \Gamma(V)$ such that
\begin{gather*}
\sum\limits_{k=1}^n \lprod{\End(V) \tensor \aA}{s_k, s_k} = \one
\end{gather*} (see Lemma~\ref{lem:diagonalize}). Now put \smash{$\hH := \C^n$} and denote by $s \in \End(\hH, V) \tensor \aA$ the element with columns $s_1, \dots, s_n$ in the canonical orthonormal basis. Then $\pi \alpha(s) = s \tensor \one_\aG$, since $s_k \in \Gamma(V)$, and further
\begin{gather*}
ss^* = \sum_{k=1}^n \lprod{\End(V) \tensor \aA}{s_k,s_k} = \one.\tag*{\qed}
\end{gather*}\renewcommand{\qed}{}
\end{proof}

\begin{Remark}
	For each representation $\pi$ of $\aG$ there is a minimal dimension, say $n(\pi)$, that the Hilbert space~$\hH$ in Lemma~\ref{lem:freeness} can take. Clearly we have $n(1) = 1$, $n(\pi \oplus \rho) \le n(\pi) + n(\rho)$, and $n(\pi \tensor \rho) \le n(\pi) \cdot n(\rho)$, using a variant of the multiplication map~$m_{\pi, \rho}$.
\end{Remark}

Suppose we f\/ix a Hilbert space $\hH_\pi$ and a respective coisometry $s(\pi)$ for each irreducible representation $\pi \in \hat \aG$. Then we may extend $\pi \mapsto \hH_\pi$ to an additive functor and \mbox{$\pi \mapsto s(\pi)$} to a family of coisometries that satisf\/ies the condition in Lemma~\ref{lem:freeness} and behaves naturally with respect to intertwiners. However, the functor $\pi \mapsto \hH_\pi$ is in general not a tensor functor and $s(\pi \tensor \rho)$ has no immediate relation to $s(\pi)$ and $s(\rho)$.

\looseness=-1 In the remaining part of this section we present a bouquet of examples. To begin with, we recall that Def\/inition~\ref{def:freeness} actually extends the classical notion of free actions of compact groups. In fact, given a compact space $P$ and a compact group $G$, it is a consequence of \cite[Proposition~7.1.12 and Theorem~7.2.6]{Phi87} that a continuous group action $\sigma\colon P\times G\rightarrow P$ is free, i.e., its stabilizer groups vanish at each point, if and only if the induced ${\rm C}^*$-dynamical system $(C(P),G,\alpha_{\sigma})$ is free in the sense of Def\/inition~\ref{def:freeness}. Therefore, Def\/inition~\ref{def:freeness} also provides a~natural framework for noncommutative principal bundles. Furthermore, we would like to point out that Def\/inition~\ref{def:freeness} characterizes classical free group actions in terms of associated vector bundles and the condition therein means that the associated vector bundles have non-degenerate f\/ibres (see, e.g.,~\cite{Wa12}).

\begin{Example} 	\label{expl:Connes-Landi}
	We would like to recall a ${\rm C}^*$-algebraic version of the nontrivial Hopf--Galois extension studied in \cite{LaSu05} (see also \cite{CoLa01}). Let $\theta \in \R$ be f\/ixed and let $\theta'$ be the skewsymmetric $4 \times 4$-matrix with $\theta_{1,2}' = \theta_{3,4}' = 0$ and $\theta'_{1,3} = \theta'_{1,4} = \theta_{2,3}' = \theta'_{2,4} = \theta/2$. We consider the universal unital ${\rm C}^*$-algebra $\alg A(\mathbb S_{\theta'}^7)$ generated by normal elements $z_1, \dots, z_4$ satisfying the relations
	\begin{gather*}	
		z_i z_j = e^{2\pi\imath \theta'_{i,j}}  z_j z_i,\qquad z_j^* z_i = e^{2\pi \imath \theta'_{i,j}} z_i z_j^*,\qquad 	\sum_{k=1}^4 z_k^* z_k^{} = \one
	\end{gather*}
	for all $1 \le i,j \le 4$. A few moments thought show that the group $G =\SU(2)$ acts strongly continuously on $\alg A(\mathbb S_{\theta'}^7)$ via the $^*$-automorphisms $(\alpha_U)_{U \in \SU(2)}$ given on generators by
	\begin{gather*}
		\alpha_U\colon  (z_1, \dots, z_4) \mapsto (z_1, \dots, z_4) \begin{pmatrix} U & 0 \\ 0 & U \end{pmatrix}.
	\end{gather*}
	Moreover, the f\/ixed point algebra turns out to be the universal unital ${\rm C}^*$-algebra $\alg A(\mathbb S_\theta^4)$ generated by normal elements $w_1, w_2$ and a self-adjoint element $x$ satisfying
	\begin{gather*}
w_1 w_2 = e^{2\pi\imath\theta}  w_2 w_1, \qquad w_2^* w_1 = e^{2\pi\imath \theta} w_1 w_2^*, \qquad \text{and} \qquad w_1^* w_1 + w_2^* w_2 + x^*x = \one.
	\end{gather*}
For $\theta=0$ all algebras are commutative and we recover the classical 7-dimensional Hopf f\/ibration of the 4-sphere, which is a well-known example of a non-trivial principal bundle. Many arguments from the classical case can be extended to arbitrary $\theta$. In particular, it is easily checked that for the fundamental 2-dimensional representation $(\pi_1,\mathbb{C}^2)$ of $\SU(2)$ a coisometry \smash{$s \in \End(\C^4, \C^2) \tensor \alg A(\mathbb S_{\theta'}^7)$} with $U \alpha_U(s) = s$ for all $U \in \SU(2)$ is given by
	\begin{gather*}
		s := \begin{pmatrix}
			z_1^* & -z_2 & z_3^* & -z_4
			\\
			z_2^* & z_1 & z_4^* & z_3
		\end{pmatrix}.
	\end{gather*}
Since every irreducible representation of $\SU(2)$ can be obtained as a subrepresentation of a~suitable tensor powers of $\pi_1$, we may take tensor products of $s$ with itself in order to f\/ind a~suitable coisometry for every representation $\pi$ of $\SU(2)$. We conclude that the compact ${\rm C}^*$-dynamical system $\bigl( \alg A(\mathbb S_{\theta'}^7), \SU(2), \alpha \bigr)$ is free.
\end{Example}

\begin{Example}\label{expl:free and ergodic}\quad
\begin{enumerate}\itemsep=0pt
	\item
	Bichon, De Rijdt and Vaes introduce in \cite{BiRiVa06} the notion of quantum multiplicity of an irreducible representation in an ergodic action of a compact quantum group and classify ergodic actions of so-called full quantum multiplicity in terms of unitary f\/iber functors. It follows from \cite[Theorem~3.9]{BiRiVa06} that these actions are free.
	\item
	According to \cite[Corollary~5.8]{BiRiVa06}, for suf\/f\/iciently small parameters $q$ the compact quantum group~$\SU_q(2)$ admits an ergodic action of full quantum multiplicity such that the multiplicity of the fundamental representation is arbitrarily large. Hence, there are plenty of free and ergodic actions of $\SU_q(2)$.
\end{enumerate}
\end{Example}

\begin{Example} \label{ex:Cuntz}
	Let $\aG$ be an $R^+$-deformation (see, e.g., \cite[Theorem~2.1]{Ban99}) of a semisimple compact Lie group. Furthermore, let $(\pi,\C^d)$ be a faithful representation of $\aG$. Then \mbox{\cite[Proposition~7.3]{Gabriel2014}} implies that the induced action $\alpha$ of $\aG$ on the Cuntz algebra $\mathcal{O}_d$ def\/ined by
\begin{gather*}
\alpha(S_i):= \sum\limits^d_{j=1} S_j \otimes \pi_{j,i}
\end{gather*} is free, where $S_1,\ldots S_d$ denote the generators of $\mathcal{O}_d$. It is not hard to check that for the representation $(\pi,\C^d)$ a coisometry $s \in \End(\C, \C^d) \tensor \mathcal{O}_d$ with $\pi_{13} \id \tensor \alpha(s) = s \otimes \one_\aG$ is given by
	\begin{gather*}
		s := \begin{pmatrix}
			S_1^* , S_2^* , \ldots , S_d^*
		\end{pmatrix}^\top.
	\end{gather*}
\end{Example}

\section{Factor systems}\label{section4}

We have seen in Lemma~\ref{lem:freeness} that freeness of a compact ${\rm C}^*$-dynamical system $(\aA, \aG, \alpha)$ can be cast in form of a family of coisometries. These coisometries may be used to give a more explicit picture of the spectral subspaces of the ${\rm C}^*$-dynamical system. In fact, let $(\pi, V)$ be a~representation of~$\aG$ and let $s(\pi) \in \End(\hH,V) \tensor \aA$ be a coisometry with $\pi_{13} \id \tensor \alpha\bigl( s(\pi) \bigr) = s(\pi) \otimes \one_\aG$ in $\End(\hH,V) \tensor \aA \tensor \aG$. Then a few moments thought show that the multiplicity space $\Gamma(V) \subseteq V \tensor \aA$ is the range of the element $s(\pi)$, i.e., we have
\begin{gather*} 	
	\Gamma(V) = s(\pi) \big(\hH \tensor \aA^\aG\big).
\end{gather*}
The explicit form allows us to phrase the correspondence structure and the multiplicative structure among the generalized isotypic components only in terms of the f\/ixed point algebra $\aA^G$ and the quantum group $\aG$. This fact was already exploited in the previous part of this series~\cite{SchWa16}, where we carried out the analysis in the case of cleft dynamical systems with a classical compact group. With some adjustments we generalize the construction here to arbitrary free ${\rm C}^*$-dynamical systems and quantum groups.

We start with a free compact ${\rm C}^*$-dynamical system $(\aA, \aG, \alpha)$ and we write brief\/ly $\aB := \aA^\aG$ for the corresponding f\/ixed point algebra. Furthermore, we choose a functorial version of the f\/inite-dimensional Hilbert spaces $\hH_\pi$ and the coisometries $s(\pi)$ for each representation $\pi$ of $\aG$ (see also the discussion after Lemma~\ref{lem:freeness}). In particular, we assume without loss of generality that $\hH_1 = \C$ and $s(1) = \one_\aB$. Then we consider for each representation $\pi$ of~$\aG$ the $^*$-homomorphism
\begin{gather*}
	\gamma_\pi\colon \aB \to \End\bigl(\hH_\pi \bigr) \tensor \aB,
	\qquad
	\gamma_\pi(b) := s(\pi)^* (\one_{V_\pi} \tensor b) s(\pi)
\end{gather*}
and for each pair $\pi, \rho$ of representations of $\aG$ the element
\begin{gather*}
	\omega(\pi, \rho) := s(\pi \tensor \rho)^* s(\pi) s(\rho) \in \End(\hH_\pi \tensor \hH_\rho, \hH_{\pi \tensor \rho}) \tensor \aB,
\end{gather*}
where $s(\pi)$ and $s(\rho)$ are amplif\/ied to act trivially on $\hH_\rho$ and $\hH_\pi$, respectively.

\begin{defn}
	Let $(\aA, \aG, \alpha)$ be a free compact ${\rm C}^*$-dynamical system. Then the system $(\hH, \gamma, \omega) = \bigl( \hH_\pi, \gamma_\pi, \omega(\pi, \rho) \bigr)_{\pi, \rho \in \hat \aG}$ constructed above is called a \emph{factor system} of $(\aA, \aG, \alpha)$.
\end{defn}

\begin{Remark} 	\label{rmk:factor_sys_basis}
	For some computations it is convenient to express the factor system in terms of f\/ixed orthonormal bases of the Hilbert spaces $\hH_\pi$, $\pi \in \hat\aG$. In this situation we denote by $s(\pi)_1, \dots, s(\pi)_n \in \Gamma(V)$ the columns of $s(\pi)$. Then the $^*$-homomorphism $\gamma_\pi\colon \aB \to M_n \tensor \aB$ has the coef\/f\/icients
	\begin{gather*}
		\gamma_\pi(b)_{i,j} = \rprod{\aB}{s(\pi)_i, b \acts s(\pi)_j}
	\end{gather*}
	for all $1 \le i \le \dim \hH_\pi$ and $1 \le j \le \dim \hH_\rho$.
	For the partial isometry $\omega(\pi,\rho)$ we f\/irst f\/ix an irreducible subrepresentation $\sigma$ of $\pi \tensor \rho$. Then the coef\/f\/icients on the corresponding subspace $\hH_\sigma\subseteq \hH_\pi \tensor \hH_\rho$ are given by
	\begin{gather*}
		\omega(\pi, \rho)_{(i,j), k} = \rprod{\aB}{m(s(\pi)_i \tensor s(\rho)_j, s(\sigma)_k}
	\end{gather*}
	for all $1 \le i \le \dim \hH_\pi$, $1 \le j \le \dim \hH_\rho$, and $1 \le k \le \dim \hH_\sigma$.
\end{Remark}

Of course, dif\/ferent choices of Hilbert spaces $\hH_\pi$ and coisometries $s(\pi)$ give rise to dif\/ferent factor systems. However, as the following lemma shows, those choices only ef\/fect the factor system by a conjugacy with partial isometries:

\begin{Lemma}[cf.\ Lemma~5.5 and Theorem~5.6 in \cite{SchWa16}]	\label{lem:factor_sys}
	For a factor system $(\hH, \gamma, \omega)$ of a free compact ${\rm C}^*$-dynamical system $(\aA, \aG, \alpha)$ with fixed point algebra $\aB$ the following assertions hold:
	\begin{enumerate}\itemsep=0pt
	\item[$1.$] We have $\omega(1,1) = \one_\aB$, $\gamma_1 = \id_\aB$ and
		\begin{gather}		\label{eq:ranges_sys}
			\omega(\pi, \rho) \omega(\pi, \rho)^*= \gamma_{\pi \tensor \rho}(\one), \qquad
			\omega(\pi, \rho)^* \omega(\pi, \rho)		= \id \tensor \gamma_\rho \bigl( \gamma_\pi(\one) \bigr),\\
			\label{eq:coaction_sys}	\gamma_{\pi \tensor \rho}(b) \omega(\pi, \rho)=	\omega(\pi, \rho) \gamma_\rho \bigl( \gamma_\pi(b) \bigr),\\
			\label{eq:cocycle_sys}	\omega(\pi, \rho \tensor \sigma) \bigl( \one \tensor \omega(\rho, \sigma) \bigr)=\omega(\pi \tensor \rho, \sigma) \id \tensor \gamma_\sigma \bigl( \omega(\pi, \rho) \bigr)
		\end{gather}
		for all representations $\pi, \rho$ of $\aG$ and $b \in \aB$. We refer to the equation~\eqref{eq:coaction_sys} as the coaction condition and to equation~\eqref{eq:cocycle_sys} as the cocycle condition.
	\item[$2.$] Let $(\hH', \gamma', \omega')$ be another factor system of $(\aA, \aG, \alpha)$. Then there is a family of partial iso\-met\-ries $v(\pi) \in \End(\hH_\pi', \hH_\pi) \tensor \aB$, $\pi \in \hat \aG$, such that
		\begin{gather}	\label{eq:conj1}v(\pi)  v(\pi)^* = \gamma_\pi(\one), \qquad		v(\pi)^* v(\pi) =\gamma_\pi'(\one),	\\
			\label{eq:conj2}v(\pi)  \gamma_\pi'(b)=		\gamma_\pi(b) \, v(\pi),\\
			\label{eq:conj3}v(\pi \tensor \rho) \omega'(\pi, \rho)=	\omega(\pi, \rho) \id \tensor \gamma_\rho \bigl( v(\pi) \bigr) \bigl( \one \tensor v(\rho) \bigr)
		\end{gather}
		hold for all $\pi, \rho \in \hat \aG$ and $b \in \aB$.
	\item[$3.$] Conversely, let $v(\pi) \in \End(\hH_\pi', \hH_\pi) \tensor \aB$, $\pi \in \hat \aG$, be a family of partial isometries for finite-dimensional Hilbert spaces $\hH_\pi'$ such that $v(\pi) v(\pi)^* = \gamma_\pi(\one)$ holds for each $\pi \in \hat \aG$. Then the following system $(\hH', \gamma', \omega')$ is a factor system of $(\aA, \aG, \alpha)$:
		\begin{gather*}
			\gamma_\pi'(b) := v(\pi)^* \gamma_\pi(b)  v(\pi),\\
			\omega'(\pi, \rho):=v(\pi \tensor \rho)^* \; \omega(\pi, \rho) \id \tensor \gamma_\rho \bigl( v(\pi) \bigr) \bigl( \one \tensor v(\rho) \bigr)
		\end{gather*}
		for all $\pi, \rho \in \hat \aG$ and $b \in \aB$.
	\end{enumerate}
\end{Lemma}

\begin{proof}For sake of a concise notation we amplify all elements to a common domain specif\/ied by the context. Let $s(\pi) \in \End(\hH_\pi, V_\pi) \tensor \aA$, $\pi \in \hat \aG$, be the coisometries with $\pi \alpha \bigl( s(\pi) \bigr) = s(\pi) \tensor \one_\aG$ that generate the factor system $(\hH, \gamma,\omega)$.

1.~Let $\pi$, $\rho$ be representations of $\hat \aG$. Using the coisometry property of $s(\pi)$, $s(\rho)$, and $s(\pi \tensor \rho)$ we obtain for the range and cokernel projection of $\omega(\pi, \rho)$
		\begin{gather*}
			\omega(\pi, \rho) \omega(\pi, \rho)^*= s(\pi \tensor \rho)^* s(\pi) s(\rho) s(\rho)^* s(\pi)^* s(\pi \tensor \rho)\\
\hphantom{\omega(\pi, \rho) \omega(\pi, \rho)^*}{} = s(\pi \tensor \rho)^* s(\pi \tensor \rho)	= \gamma_{\pi \tensor \rho}(\one),	\\
			\omega(\pi, \rho)^*  \omega(\pi, \rho)= s(\rho)^* s(\pi)^* s(\pi \tensor \rho) s(\pi \tensor \rho)^* s(\pi) s(\rho)	\\
\hphantom{\omega(\pi, \rho)^*  \omega(\pi, \rho)}{} = s(\rho)^* s(\pi)^* s(\pi) s(\rho)	= \gamma_\rho( \gamma_\pi(\one)).
		\end{gather*}	
To show the other two asserted equations we compute the left and right hand side individually using the coisometry property and compare for all $b \in \aA^\aG$:
		\begin{gather*}
			\omega(\pi, \rho) \gamma_\rho \bigl( \gamma_\pi(b) \bigr)= s(\pi \tensor \rho)^* s(\pi) s(\rho) s(\rho)^* s(\pi)^*  b s(\pi) s(\rho)\\
\hphantom{\omega(\pi, \rho) \gamma_\rho \bigl( \gamma_\pi(b) \bigr)}{} = s(\pi \tensor \rho)^* b s(\pi) s(\rho),\\
			\gamma_{\pi \tensor \rho}(b) \omega(\pi, \rho)= s(\pi \tensor \rho)^*  b s(\pi \tensor \rho) s(\pi \tensor \rho)^* s(\pi) s(\rho)	\\
\hphantom{\gamma_{\pi \tensor \rho}(b) \omega(\pi, \rho)}{} = s(\pi \tensor \rho)^* b s(\pi) s(\rho),\\
\omega(\pi, \rho \tensor \sigma) \omega(\rho, \sigma)= s(\pi \tensor \rho \tensor \sigma)^* s(\pi) s(\rho \tensor \sigma) s(\rho \tensor \sigma)^* s(\rho) s(\sigma)	\\
\hphantom{\omega(\pi, \rho \tensor \sigma) \omega(\rho, \sigma)}{} = s(\pi \tensor \rho \tensor \sigma)^* s(\pi) s(\rho) s(\sigma),	\\
\omega(\pi \tensor \rho, \sigma) \gamma_\sigma \bigl( \omega(\pi, \rho) \bigr)= s(\pi \tensor \rho \tensor \sigma)^* s(\pi \tensor \rho) s(\sigma) s(\sigma)^* s(\pi \tensor \rho)^* s(\pi) s(\rho) s(\sigma)	\\
\hphantom{\omega(\pi \tensor \rho, \sigma) \gamma_\sigma \bigl( \omega(\pi, \rho) \bigr)}{} = s(\pi \tensor \rho \tensor \sigma)^* s(\pi) s(\rho) s(\sigma).
		\end{gather*}

2.~Let $s'(\pi) \in \End(\hH'_\pi, V'_\pi) \tensor \aA$, $\pi \in \hat \aG$, be the coisometries with $\pi \alpha\bigl(s'(\pi) \bigr) = s'(\pi) \tensor \one_\aG$ that generate the factor system $(\hH', \gamma', \omega')$. Then the coisometry property implies that for each $\pi \in \hat G$ the element $v(\pi) := s(\pi)^* \cdot s'(\pi)$ is a partial isometry satisfying $v(\pi) v(\pi)^* = s(\pi)^* s(\pi) = \gamma_\pi(\one)$ and $v(\pi)^* v(\pi) = s'(\pi)^* s'(\pi) = \gamma_\pi'(\one)$. Similarly the asserted relation of the $^*$-homomor\-phisms~$\gamma_\pi$ and $\gamma_\pi'$ and of the elements $\omega(\pi, \rho)$ and $\omega'(\pi, \rho)$ immediately follow from the coisometry property.
\end{proof}

Next, we explain how the correspondence structure of the isotypic components of a free compact ${\rm C}^*$-dynamical system can be expressed only in terms of quantities of an associated factor system. For this purpose, let $(\hH, \gamma, \omega)$ be a factor system of a free compact ${\rm C}^*$-dynamical system $(\aA, \aG, \alpha)$ with f\/ixed point algebra $\aB$. Then, for a representation $(\pi,V)$ of $\aG$, the left and right action of $\aB$ and the inner product on $\Gamma(V)$ are given by
\begin{gather*}
	b \acts (s(\pi) x) = s(\pi) \gamma_\pi (b) x,\\
	(s(\pi)  x)\acts b = s(\pi)  x  b,\\
	\scal{s(\pi)  x,s(\pi) y}_\aB = \rprod{\aB}{x,  \gamma_\pi(\one_\aB) y}
\end{gather*}
for all $b \in \aB$ and $x,y \in \hH_\pi \tensor \aB$. Moreover, for two representation $(\pi,V)$ and $(\rho,W)$ of~$\aG$ the multiplication map $m_{\pi, \rho}\colon  \Gamma(V) \tensor_\aB \Gamma(W) \to \Gamma(V \tensor W)$ can be written as
\begin{gather*}	
	m_{\pi, \rho} \bigl( s(\pi) x \tensor s(\rho) y \bigr)= s(\pi \tensor \rho)  \omega(\pi,\rho)  \gamma_\rho(x)  y
\end{gather*}
for all $x \in \hH_\pi \tensor \aB$ and $y \in \hH_\rho \tensor \aB$, where $\gamma_\rho(x) y$ is given by the linear extension of
\begin{gather*} \gamma_\rho(v \tensor b_1) (w \tensor b_2) = v \tensor \bigl( \gamma_\rho(b_1)(w \tensor b_2) \bigr)
\end{gather*}
for all $v \in \hH_\pi$, $w \in \hH_\rho$, and $b_1, b_2 \in \aB$. As a consequence, up to equivalence, the ${\rm C}^*$-dynamical system is uniquely determined by its factor system and vice versa. More precisely, we say that two factor systems $(\hH, \gamma, \omega)$ and $(\hH', \gamma', \omega')$ are \emph{conjugated} if there is a family of partial isometries $v(\pi) \in \End(\hH_\pi', \hH_\pi) \tensor \aB$, $\pi \in \hat \aG$, satisfying the equations~\eqref{eq:conj1}, \eqref{eq:conj2}, and \eqref{eq:conj3} for all $\pi, \rho, \sigma \in \hat\aG$ and $b \in \aB$. Then we have the following 1-to-1 correspondence:

\begin{Theorem} \label{thm:equivalence}	Let $(\aA, \aG, \alpha)$ and $(\aA', \aG, \alpha')$ be free compact ${\rm C}^*$-dynamical systems with the same fixed point algebra $\aB$ and let $(\hH, \gamma, \omega)$ and $(\hH', \gamma', \omega')$ be associated factor systems, respectively.
	Then the following statements are equivalent:
	\begin{enumerate}\itemsep=0pt
	\item[$(a)$]	
		The ${\rm C}^*$-dynamical systems $(\aA, \aG, \alpha)$ and $(\aA', \aG, \alpha')$ are equivalent.
	\item[$(b)$]	
		The factor systems $(\hH, \gamma, \omega)$ and $(\hH', \gamma', \omega')$ are conjugated.
	\end{enumerate}
\end{Theorem}
\begin{proof}
	As a distinction we add a prime to all notions referring to $(\aA', \aG, \alpha')$.

1.~To prove that (a) implies (b) it suf\/f\/ices to show that for the same ${\rm C}^*$-dynamical system dif\/ferent choices of coisometries lead to conjugated factor systems. This is exactly the second statement of Lemma~\ref{lem:factor_sys}.

2.~For the converse implication let $s(\pi) \in \End(\hH_\pi, V_\pi) \tensor \aA$, $\pi \in \hat \aG$, be the coisometries with $\pi_{13} \id \tensor \alpha \bigl( s(\pi) \bigr) = s(\pi) \tensor \one_\aG$ that generate the factor system $(\hH, \gamma,\omega)$, and likewise $s'(\pi) \in \End(\hH'_\pi, V_\pi) \tensor \aA$ for $(\hH', \gamma', \omega')$. Furthermore, let $v(\pi)$, $\pi \in \hat \aG$, be the partial isometries which realize the conjugation of the factor systems. Then a few moments thought show that, due to equations~\eqref{eq:conj1} and \eqref{eq:conj2}, for every representation $(\pi,V)$ of~$\aG$ the map
		\begin{gather*}
			\phi_\pi\colon \ \Gamma'(V) \to \Gamma(V), 	\qquad 	s'(\pi)  x \mapsto s(\pi) v(\pi) x
		\end{gather*}
for all $x \in \hH'_\pi \tensor \aB$ is a well-def\/ined isomorphism of correspondences of $\aB$. Moreover, by equation~\eqref{eq:conj3}, these isomorphisms intertwine the multiplication maps, i.e., we have
\begin{gather*}
m_{\pi, \rho} \bigl( \phi_\pi(x) \tensor \phi_\rho(y) \bigr) = \phi_{\pi \tensor \rho} \bigl( m_{\pi, \rho}'(x \tensor y) \bigr)
\end{gather*}
for all representations $(\pi,V), (\rho,W) \in \hat\aG$ and all elements $x \in \Gamma'(V)$ and $y \in \Gamma'(W)$. Since $(\aA, \aG, \alpha)$ can be reconstructed from the correspondences $\Gamma(V)$ and the multiplicative structure between them (cf.\ Lemma~\ref{lem:unitary tensor functor} or \cite[Section~2]{Ne13}), and likewise for $(\aA', \aG, \alpha')$ with $\Gamma'(V)$, it is now easily checked that the maps $\phi_\pi$, $\pi \in \hat \aG$, give rise to an equivalence between $(\aA, \aG, \alpha)$ and $(\aA', \aG, \alpha')$ (cf.\ also \cite[Theorem~5]{SchWa16}).
\end{proof}

A particular simple class of free actions are so-called cleft actions (see~\cite{SchWa16}). Regarded as noncommutative principal bundles, these actions are characterized by the fact that all associated noncommutative vector bundles are trivial. For convenience of the reader we recall the def\/inition.

\begin{defn} \label{def:cleft}
	A compact ${\rm C}^*$-dynamical system $(\aA, \aG, \alpha)$ is called \emph{cleft} if for each irreducible representation $(\pi,V)$ of $\aG$ the so-called generalized isotypic component
	\begin{gather*}
		A_2(\pi) := \{ x \in \End(V) \tensor \aA \,|\, \pi_{13} \id \tensor \alpha(x) = x \tensor \one_\aG \} \subseteq \End(V) \tensor \aA
	\end{gather*}
	contains a unitary element. It directly follows from Lemma~\ref{lem:freeness} that cleft ${\rm C}^*$-dynamical systems are free.
\end{defn}

\begin{Example}	Given a unital ${\rm C}^*$-algebra $\aB$ and a compact quantum group $\aG$, the most basic example of a cleft action is given by the ${\rm C}^*$-dynamical system $\bigl( \aB \tensor \aG, \aG, \id \tensor \Delta)$. In fact, for any irreducible representation $(\pi, V)$ of $\aG$ the unitary $U:=\pi^*_{13} \in \End(V) \tensor \aB \tensor \aG$ satisf\/ies $\pi_{13} \, \id \tensor \Delta(U) = U \otimes \one_\aG$.
\end{Example}

\begin{Example}	For $\aG = \SU_q(2)$ the only cleft and ergodic action is the canonical action of $\SU_q(2)$ on itself (see \cite[Corollary~5.9]{BiRiVa06}). For $q=1$ this already follows from the seminal work of Wassermann~\cite{Wass88b}.
\end{Example}

\begin{Example}[cf.\ Example~\ref{expl:free and ergodic}]	For an arbitrary compact quantum group, the authors of~\cite{BiRiVa06} provide a classif\/ication of unitary f\/iber functors which preserve the dimension in terms of unitary 2-cocycles on the dual quantum group. It is not hard to see that the corresponding actions are cleft.
\end{Example}

\begin{Example}It can be shown that the free ${\rm C}^*$-dynamical system $\bigl( \alg A(\mathbb S_{\theta'}^7), \SU(2), \alpha \bigr)$ discussed in Example~\ref{expl:Connes-Landi} is not cleft (cf.\ \cite[Proposition~9]{LaSu05}).
\end{Example}

We continue with a characterization of cleft actions in terms of their factor systems. For this we recall that two projections $p \in \End(V) \tensor \aB$ and $q \in \End(W) \tensor \aB$ with f\/inite-dimensional Hilbert spaces $V,W$ are called \emph{Murray--von~Neumann equivalent} over $\aB$ if there is a partial isometry $v \in \End(V,W) \tensor \aB$ satisfying $p = v^* v$ and $q = vv^*$.

\begin{Lemma}	\label{lem:cleft}	Let $(\aA, \aG, \alpha)$ be a free compact ${\rm C}^*$-dynamical system with fixed point algebra~$\aB$. Then the following statements are equivalent:
	\begin{enumerate}\itemsep=0pt
	\item[$(a)$] The system $(\aA, \aG,\alpha)$ is cleft.
	\item[$(b)$] For some and hence for every factor system $(\hH, \gamma, \omega)$ and every $(\pi,V) \in \hat \aG$ the projection $\gamma_\pi(\one_\aB)$ is Murray--von~Neumann equivalent to $\one_V \tensor \one_\aB$ over $\aB$.
	\end{enumerate}
\end{Lemma}
\begin{proof}1.~If $(\aA, \aG, \alpha)$ is cleft, each $A_2(\pi)$, $\pi \in \hat \aG$, contains a unitary element $s(\pi)$. A~factor system $(\hH, \gamma, \omega)$ is then given by $\hH_\pi = V_\pi$ and $\gamma_\pi(x) = s(\pi)^* (x \tensor \one) s(\pi)$ for all $\pi \in \hat \aG$. In particular, we have $\gamma_\pi(\one_\aB) = s(\pi)^* s(\pi) = \one_V \tensor \one_\aB$. By Lemma~\ref{lem:factor_sys} every other factor system $(\aA, \aG, \alpha)$ dif\/fers only by partial isometries in a respective amplif\/ication and therefore satisf\/ies the same relation.

2.~Conversely, suppose that $(\hH, \gamma, \omega)$ is a factor system of $(\aA, \aG, \alpha)$ such that for every \smash{$(\pi, V) \in \hat G$} the projections $\gamma_\pi(\one_\aB)$ and $\one_V \tensor \one_\aB$ are Murray--von~Neumann equivalent. That is, we may f\/ind partial isometries $v(\pi) \in \End(V_\pi, \hH_\pi)$, $\pi \in \hat \aG$, such that $\gamma_\pi(\one_\aB) = v(\pi) v(\pi)^*$ and $\one_\pi \tensor \one_\aB = v(\pi)^* v(\pi)$. By conjugating the factor system with this family of partial isometries we may assume that $\hH_\pi = V_\pi$ and $\gamma_\pi(\one_\aB) = \one_\pi \tensor \one_\aB$. Moreover, for the factor system we may pick a~family of coisometries $s(\pi) \in A_2(\pi)$, $\pi \in \hat \aG$, with $\gamma_\pi(b) = s(\pi)^*(\one_\pi \tensor b) s(\pi)$ for all $b \in \aB$. Then we have $\one_\pi \tensor \one_\aB = \gamma_\pi(\one_\aB) = s(\pi)^* s(\pi)$, that is, $s(\pi)$ is unitary.
\end{proof}

\begin{Example}Suppose we are in the context of Example~\ref{ex:Cuntz} with $\aG=\SU_q(2)$ and the natural representation $(\pi,\mathbb{C}^2)$ of $\SU_q(2)$. Furthermore, let $\aB$ be the f\/ixed point algebra of the induced free compact ${\rm C}^*$-dynamical system $(\mathcal{O}_2,\SU_q(2),\alpha)$. Then a few moments thought show that the $^*$-homomorphism $\gamma_{\pi} \colon  \aB \rightarrow \aB$ induced by the coisometry $s=(S_1^*,S_2^*)^\top$ satisf\/ies $\gamma_{\pi}(\one_\aB)=\one_\aB$. Since \cite[Proposition~6.10]{Gabriel2014} implies that $K_0(\aB)$ can be identif\/ied with the integers in such a way that $[\one_\aB]=1$ (see also \cite{GabWeb16,KoYuNaMaWaYa92,Marc98}), it follows from Lemma~\ref{lem:cleft} that $(\mathcal{O}_2,\SU_q(2),\alpha)$ is not cleft.
\end{Example}

\section{Construction of free actions}\label{section5}

In the previous section we have seen that a free compact ${\rm C}^*$-dynamical system is uniquely determined by its factor system $(\hH, \gamma, \omega)$ and under which equivalence relation this becomes 1-to-1 correspondence (Theorem~\ref{thm:equivalence}). In this section we will show that in fact every factor system $(\hH, \gamma, \omega)$ satisfying the algebraic relations of Lemma~\ref{lem:factor_sys} gives rise to a free compact ${\rm C}^*$-dynamical system. The construction is based on the fact, that the factor system $(\hH, \gamma, \omega)$ allows us to completely reconstruct the correspondence structure of the multiplicity spaces $\Gamma(V)$ and their multiplicative structure, i.e., the factor system provides a unitary tensor functor $V \mapsto \Gamma(V)$ and hence a compact ${\rm C}^*$-dynamical system (see~\cite{CoYa13b,Ne13}). We recall the major steps in order to show that this construction yields a free compact ${\rm C}^*$-dynamical system with factor system $(\hH, \gamma, \omega)$.

Throughout the following let $\aB$ be a f\/ixed unital ${\rm C}^*$-algebra and let $\aG$ be a f\/ixed reduced compact quantum group. Furthermore, let $(\hH, \gamma, \omega) = \bigl( \hH_\pi, \gamma_\pi, \omega(\pi,\rho) \bigr)_{\pi, \rho\in \hat \aG}$ be a family of f\/inite-dimensional Hilbert spaces $\hH_\pi$, $^*$-homomorphisms $\gamma_\pi\colon \aB \to \End(\hH_\pi) \tensor\aB$, and partial isometries $\omega(\pi,\rho) \in \End(\hH_\pi \tensor \hH_\rho, \hH_{\pi\tensor\rho}) \tensor \aB$. By taking direct sums of irreducible representations, we def\/ine $\hH_\pi$, $\gamma_\pi$ and $\omega(\pi, \rho)$ for arbitrary representations $\pi, \rho$ of~$\aG$. In particular, for each intertwiner $T\colon V_\pi \to V_\rho$ we have a linear map $\hH(T)\colon \hH_\pi \to \hH_\rho$.

\begin{defn} \label{defn:factor_sys} A system $(\hH, \gamma, \omega)$ as described above is called a \emph{factor system} for the pair $(\aB,\aG)$ if it satisf\/ies equations~\eqref{eq:ranges_sys}, \eqref{eq:coaction_sys}, \eqref{eq:cocycle_sys} for all $\pi, \rho \in \hat \aG$ and $b \in \aB$, and if the normalization conditions $\hH_1 = \C$, $\gamma_1 = \id_\aB$, $\omega(1,1) = \one_\aB$ holds.
\end{defn}

From now on we suppose that $(\hH, \gamma, \omega)$ is a factor system. Then, for each representation $(\pi, V)$ of $\aG$, we consider the vector space
\begin{gather} 	\label{eq:Gamma}
	\Gamma(V) := \gamma_\pi(\one) (\hH_\pi \tensor \aB).
\end{gather}
A few moments thought show that this space caries a natural right Hilbert $\aB$-module structure given by restricting the action $(v_1 \tensor b_1)\acts b_2 := v_1 \tensor b_1 b_2$ and the inner product $\rprod{\aB}{v_1 \tensor b_1,v_2 \tensor b_2} := \scal{v_1,v_2} b_1^* b_2$ for $v_1,v_2 \in \hH_\pi$ and $b_1,b_2 \in \aB$. Moreover, we equip $\Gamma(V)$ with the left action $b \acts x := \gamma_\pi(b) x$ for $b \in \aB$ and $x \in \Gamma(V)$. Then it is easily checked that $\Gamma(V)$ is a correspondence over~$\aB$ and that $V \mapsto \Gamma(V)$ becomes an additive functor from the representation category of $\aG$ into the category of ${\rm C}^*$-correspondences over $\aB$.

For each pair $(\pi,V), (\rho,W)$ of representation of $\aG$ we def\/ine a linear map
\begin{align}
		m_{\pi,\rho}\colon \ & \Gamma(V) \tensor_\aB \Gamma(W) \to \Gamma(V \tensor W),	\nonumber\\
		& m_{\pi, \rho}(x \tensor y) := \omega(\pi, \rho)  \gamma_\rho(x) y,	\label{eq:multipl}
	\end{align}
where for elementary tensors we write brief\/ly $\gamma_\rho(v \tensor b_1) (w \tensor b_2) := v \tensor \gamma_\rho(b_1)(w \tensor b_2)$ for all $v \in \hH_\pi$, $w \in \hH_\rho$, and $b_1, b_2 \in \aB$. It is easily checked that the maps $m_{\pi, \rho}$ are well-def\/ined and behave naturally with respect to intertwiners. In fact, we are going to show that $V \mapsto \Gamma(V)$ together with the maps $m_{\pi, \rho}$ forms a \emph{unitary tensor functor} in the sense of \cite[Def\/inition~2.1]{Ne13}. For convenience of the reader we recall the def\/inition in the current context:

\begin{defn}A linear functor $V \mapsto \Gamma(V)$ from the representation category of $\aG$ into the category of ${\rm C}^*$-correspondences over $\aB$ together with a~$\aB$-bilinear family of unitary maps
\begin{gather*}
m_{\pi, \rho}\colon \ \Gamma(V) \tensor_\aB \Gamma(W) \to \Gamma(V \tensor W)
\end{gather*}
for all representations $(\pi,V)$, $(\rho,W)$ of $\aG$ is called a \emph{unitary tensor functor} if the following conditions hold:
	\begin{enumerate}\itemsep=0pt
	\item For the trivial representation $(1,\C) \in \hat G$ we have $\Gamma(\C) = \aB$ and for all $(\pi,V) \in \hat G$ we have $m_{\pi,1}(x \tensor b) = x \acts b$ and $m_{1, \pi}(b \tensor x) = b \acts x$ for all $x \in \Gamma(V)$, $b \in \aB$.
	\item For every intertwiner $T\colon V \to W$ we have $\Gamma(T^*) = \Gamma(T)^*$.
	\item The maps $m$ are associative in the sense that for all $\pi, \rho, \sigma \in \hat G$ we have
		\begin{gather*}
			m_{\pi, \rho \tensor \sigma} \circ (\id \tensor m_{\rho, \sigma})= m_{\pi \tensor \rho, \sigma} \circ (m_{\pi, \rho} \tensor\id).
		\end{gather*}
\end{enumerate}
\end{defn}

\begin{Lemma} \label{lem:unitary tensor functor} The functor $V \mapsto \Gamma(V)$ and the maps $m_{\pi, \rho}\colon \Gamma(V)\tensor_\aB \Gamma(W) \to \Gamma(V \tensor W)$ given by the equations~\eqref{eq:Gamma} and~\eqref{eq:multipl}, respectively, for $(\pi,V),(\rho,W) \in \hat\aG$ constitute a unitary tensor functor.
\end{Lemma}
\begin{proof}
1.~The normalization $\Gamma(\C) = \aB$ as correspondence immediately follows from $\hH_1 = \C$ and $\gamma_1 = \id_\aB$. Moreover, the normalization $\omega(1,1) = 1$ together with conditions~\eqref{eq:ranges_sys} and~\eqref{eq:coaction_sys} of the factor system imply $\omega(\pi,1) = \one_{\hH_\pi} = \omega(1,\pi)$ that for every \mbox{$(\pi,V) \in \hat \aG$}. Hence, we obtain $m_{\pi, 1}(x \tensor b) = \omega(\pi,1) \gamma_\aB(x) b = x \acts b$ and $m_{1,\pi}(b \tensor x) = \omega(1,\pi) \gamma_\pi(b) x = b\acts x$ for all $x \in \Gamma(V)$ and $b \in \aB$.

2.~For any intertwiner $T\colon V \to W$ it is easily checked that $\hH(T^*) = \hH(T)^*$ which in turn implies that $\Gamma(T) = \hH(T) \tensor \id_\aB|_{\Gamma(V)}$ is adjointable with $\Gamma(T) = \Gamma(T)^*$.

3.~Associativity is an immediate consequence of the coaction and cocycle condition of the factor system. More precisely for all representations $\pi,\rho,\sigma \in \hat \aG$ and elements $x \in \Gamma(V_\pi)$, $y \in \Gamma(V_\rho)$, $z \in \Gamma(V_\sigma)$ we have
		\begin{gather*}
m_{\pi, \rho \tensor \sigma}\bigl(x \tensor m_{\rho,\sigma}(y \tensor z) \bigr)	= \omega(\pi, \rho \tensor \sigma) \gamma_{\rho \tensor\sigma}(x) \bigl( \omega(\rho,\sigma) \gamma_\sigma(y) z \bigr)			\\
\qquad{} \overset{\eqref{eq:coaction_sys}}=	\omega(\pi,\rho \tensor \sigma) \omega(\rho,\sigma) \gamma_\sigma \bigl( \gamma_\rho(x) y \bigr) z \overset{\eqref{eq:cocycle_sys}}=\omega(\pi \tensor \rho,\sigma) \gamma_\sigma \bigl( \omega(\pi,\rho) \bigr) \gamma_\sigma \bigl( \gamma_\rho(x) y) z\\
\qquad{} = \omega(\pi \tensor \rho, \sigma) \gamma_\sigma \bigl( \omega(\pi,\rho)\gamma_\rho(x)y \bigr) z=m_{\pi \tensor \rho, \sigma} \bigl( m_{\pi,\rho}(x \tensor y) \tensor z \bigr).
		\end{gather*}

4.~It remains to show that the maps $m_{\pi,\rho}\colon \Gamma(V) \tensor_\aB \Gamma(W) \to \Gamma(V \tensor W)$ are unitary for all representations $(\pi,V)$, $(\rho,W)$ of~$\aG$. To see that $m_{\pi, \rho}$ is isometric we observe that by equation~\eqref{eq:ranges_sys} the projection $\omega(\pi, \rho)^* \omega(\pi,\rho) = \gamma_\rho \bigl( \gamma_\pi(\one) \bigr)$ is larger than than the subspace of $\hH_\pi \tensor \hH_\rho \tensor \aB$ generated by all $\gamma_\rho(x) y$ with $x \in \Gamma(V)$, $y \in \Gamma(W)$. Hence, we have
		\begin{gather*}
\rprod{\aB}{m_{\pi,\rho}(x_1 \tensor y_1), m_{\pi,\rho}(x_2 \tensor y_2)}= \rprod{\aB}{\gamma_\rho(x_1)y_1, \omega(\pi,\rho)^*\omega(\pi,\rho) \gamma_\rho(x_2) y_1}		\\
\qquad{} = \rprod{\aB}{\gamma_\rho(x_1) y_1, \gamma_\rho(x_2) y_1}	= \rprod{\aB}{x_1 \tensor y_1, x_2 \tensor y_2}
\end{gather*}
for all $x_1, x_2 \in \Gamma(V)$ and $y_1, y_2 \in \Gamma(W)$.	To show that $m_{\pi,\rho}$ is surjective, we notice that $\Gamma(V)$ is linearly generated by all elements of the form $\gamma_\pi(\one) (v \tensor b)$ with $v \in \hH_\pi$ and $b \in \aB$; and likewise for $\Gamma(W)$. By equation~\eqref{eq:ranges_sys} the projection $\one_{\hH_\pi} \tensor \gamma_\rho(\one_\aB)$ is larger than the cokernel projection $\omega(\pi,\rho)^* \omega(\pi,\rho) = \gamma_\rho \bigl( \gamma_\pi(\one)\bigr)$. Choosing the elements $x := v \tensor \one_\aB$ and $y := w \tensor b$, we therefore f\/ind that the range of $m_{\pi,\rho}$ contains all elements of the form
\begin{gather*}
\omega(\pi,\rho) \bigl( v \tensor \gamma_\rho(\one)(w \tensor b) \bigr)	= \omega(\pi,\rho) (v \tensor w \tensor b)
\end{gather*}
with $v \in \hH_\pi$, $w\in \hH_\rho$, $b \in \aB$. Hence, the image of $m_{\pi,\rho}$ contains the range of $\omega(\pi,\rho)$, which by equation~\eqref{eq:ranges_sys} is given by $\gamma_{\pi\tensor\rho}(\one) (\hH_{\pi \tensor\rho} \tensor \aB) = \Gamma(V \tensor W)$.
\end{proof}

Having the unitary tensor functor in hands, we may construct a ${\rm C}^*$-dynamical system as presented in \cite{CoYa13b,Ne13}. For convenience of the reader we brief\/ly summarize the main steps. We consider the algebraic direct sum
\begin{gather*}
	A := \bigoplus_{(\pi,V) \in \hat \aG} V \tensor \Gamma\big(\bar V\big).
\end{gather*}
We equip each summand of this space with its canonical $\aB$-valued inner product given by \linebreak $\rprod{\aB}{v \tensor x, w \tensor y} = \scal{v,w} \scal{x,y}_\aB$ for all $v,w \in V$ and $x,y \in \Gamma(\bar V)$, and we extend the re\-sul\-ting inner product sesquilinearly to~$A$. Moreover, we equip $A$ with the multiplication def\/ined, for $\bar v \tensor x \in \bar V \tensor \Gamma(V)$ and $\bar w \tensor y \in \bar W \tensor \Gamma(W)$ with $(\pi,V), (\rho,W) \in \hat \aG$, by the product
\begin{gather*}
	(v \tensor x) \bullet (w \tensor y) := \sum_{k=1}^N \bigl( S_k^* \tensor \Gamma\big(\bar S_k\big)^* \bigr) \bigl(v \tensor w \tensor m_{\bar \pi,\bar\rho}(x \tensor y) \bigr)\in \sum_{k=1}^N V_{\sigma_k} \tensor \Gamma\big(\bar V_{\sigma_k}\big),
\end{gather*}
where $S_1, \dots, S_N$ is a complete set of isometric intertwiners $S_k\colon V_{\sigma_k} \to V \tensor W$, $\sigma_k \in \hat \aG$, with respective conjugates $\bar S_k\colon \bar V_{\sigma_k} \to \bar V \tensor \bar W$. Extending this product bilinearly yields an associative multiplication on $A$. The algebra $\aB$ can be regarded as the subalgebra of~$A$ corresponding to the trivial representation, and the left and right module action of~$\aB$ coincides with the multiplication on~$A$.

The next step is to construct an involution on $A$. For this purpose we f\/irst recall that for an irreducible representation $(\pi,V)$ of $\hat G$ there is a pair of intertwiners $R\colon \C \to V \tensor \bar V$ and $\bar R\colon  \C \to \bar V \tensor V$ such that $(R^* \tensor \id_V) (\id_V \tensor \bar R) = \id_V$. With this we may def\/ine involutions $^+\colon  \Gamma(V) \to \Gamma(\bar V)$ and $^+\colon \bar V \to V$ by putting
\begin{gather*}
	x^+ := m[x]^* \bigl( \Gamma(R)(\one_\aB) \bigr),\qquad 	\bar v^+ := i[\bar v]^* \bar R(1)
\end{gather*}
where we brief\/ly write $m[x]\colon \Gamma(\bar V) \to \Gamma(V \tensor \bar V)$ for the map $m[x](y) := m_{\pi,\bar\pi}(x \tensor y)$ and $i[\bar v]\colon V \to V \tensor \bar V$ for the map $i[\bar v](w) := \bar v \tensor w$. Then for $\bar v \tensor x \in \bar V \tensor \Gamma(V) \subseteq A$ we may put $(\bar v \tensor x)^+ := \bar v^+ \tensor x^+$ and extend this anilinearly to a map on $A$. It can be shown that this involution turns $A$ into a $^*$-algebra (see \cite[Lemma~2.5]{Ne13}).

\begin{Remark}\sloppy Our conventions for the inner products and the involution slightly deviate from~\cite{Ne13}, but the reader may easily adapt the arguments of~\cite{Ne13} to our conventions.
\end{Remark}

Every summand $\bar V \tensor \Gamma(V)$ admits a unitary representation of $\aG$ by acting on the f\/irst tensor factor. Taking direct sums yields a map $\alpha\colon A \to A \tensor \aG$. This map is in fact a $^*$-homomorphism satisfying $(\alpha \tensor \id_\aG) \circ \alpha = (\alpha \tensor \Delta) \circ \alpha$ (see \cite[Lemma~2.6]{Ne13}). Altogether we have an algebraic action of the quantum group $\aG$ on the $^*$-algebra $\aA$. From this we may pass to a ${\rm C}^*$-dynamical system by taking the completion $\hH_A$ of $A$ with respect to the norm $\norm{x}_2 := \norm{\rprod{\aB}{x,x}}^{1/2}$. Then the left multiplication of $A$ yields a faithful representation $\lambda\colon A \to \End(\hH_A)$ and a ${\rm C}^*$-algebra $\aA := \overline{\lambda(A)}$. The $^*$-homomorphism $\alpha$ can be extended to an action $\alpha\colon \aA \to \aA \tensor \aG$, which we denote by the same letter. Since we started with a unitary tensor functor, the corresponding compact ${\rm C}^*$-dynamical system $(\aA, \aG, \alpha)$ is free. For details we refer the reader to \cite[Section~4]{CoYa13b}.

\begin{Lemma}The free compact ${\rm C}^*$-dynamical system $(\aA, \aG, \alpha)$ admits $(\hH, \gamma, \omega)$ as one of its factor systems.
\end{Lemma}
\begin{proof}First note that for an irreducible representation $(\pi,V) \in \hat \aG$, the $\pi$-isotypic component of $(\aA, \aG, \alpha)$ is obviously given by $V \tensor \Gamma(\bar V)$. Hence the $\pi$-multiplicity space of the ${\rm C}^*$-dynamical system
\begin{gather*}
	\Gamma_\aA(V) := \{ x \in V \tensor \aA \,|\, \pi_{13} \id_V \tensor \alpha(x) = x \tensor \one_\aG \} \subseteq V \tensor \bar V \tensor \Gamma(V)
\end{gather*}
is isomorphic to $\Gamma(V)$ as a correspondence over $\aB$. More precisely, a few moments thought show that an isomorphism is given by $\varphi\colon \Gamma(V) \to \Gamma_\aA(V)$, $x \mapsto R(1) \tensor x$. Next, f\/ix an orthonormal basis $f_1, \dots, f_n$ of $\hH_\pi$ and consider the elements
	\begin{gather*}
		s_k := \gamma_\pi(\one) (f_k \tensor \one_\aB), \qquad 1 \le k\le n.
	\end{gather*}
Then it is easily checked that the collection of elements $s_k$ ($1 \le k\le n$) for each $\pi \in \hat \aG$ provide a factor system $(\tilde \hH, \tilde \gamma, \tilde \omega)$ of $(\aA, \aG, \alpha)$ with Hilbert spaces $\tilde \hH_\pi = \hH_\pi$. In terms of the chosen basis $f_1, \dots, f_n$, the $^*$-homomorphism $\tilde \gamma_\pi\colon  \aB \to M_n \tensor \aB$ for $\pi \in \hat \aG$ is given by
\begin{gather*}
	\tilde \gamma_\pi(b)_{i,j} = \rprod{\aB}{\varphi(s_i), b \acts \varphi(s_j)} = \rprod{\aB}{s_i, b \acts s_j}
	= \rprod{\aB}{f_i \tensor \one_\aB, \gamma_\pi(b) (f_j \tensor \one_\aB)}= \gamma_\pi(b)_{i,j}
\end{gather*}
for all $b \in \aB$ (see Remark~\ref{rmk:factor_sys_basis}). That is, we have $\tilde \gamma = \gamma$ and similar computation shows that $\tilde \omega = \omega$, too. Consequently, we f\/ind that $(\hH, \gamma, \omega)$ is indeed a factor system of the free compact ${\rm C}^*$-dynamical system $(\aA, \aG, \alpha)$.
\end{proof}

Summarizing the previous results, we have thus proved our main theorem:

\begin{Theorem} \label{thm:main_class_thm}	Let $\aB$ be a unital ${\rm C}^*$-algebra and let $\aB$ be a compact quantum group. Then there is a one-to-one correspondence between the set of equivalence classes of free ${\rm C}^*$-dynamical systems with fixed point algebra $\aB$ and compact quantum group $\aG$ and the set of conjugacy classes of factor systems for $(\aB,\aG)$.
\end{Theorem}

\section{Coverings of the noncommutative 2-torus}\label{section6}

Given a unital ${\rm C}^*$-algebra $\aB$, we call a free compact ${\rm C}^*$-dynamical system $(\aA,\aG,\alpha)$ with a~f\/i\-ni\-te quantum group $\aG$ and f\/ixed point algebra $\aB$ a \emph{finite covering} of $\aB$. The main purpose of this section is to use factor systems to show that f\/inite coverings of generic irrational rotation ${\rm C}^*$-al\-gebras are cleft (cf.\ Def\/inition~\ref{def:cleft}).

\begin{Lemma}	\label{lem:pos_homom}
	Let $\theta \in \R$. Then every positive group homomorphism of $\Z + \theta \Z$ is a multiple of the identity.
\end{Lemma}
\begin{proof}
	Let $h\colon \Z + \theta \Z \to \Z + \theta \Z$ be a positive group homomorphism. Then for all $x,y \in \Z$ we have that $x + \theta y\ge 0$ implies $h(1) x + h(\theta) y \ge 0$ and $x + \theta y \le 0$ implies $h(1) x + h(\theta) y \le 0$. Considering $q := -x/y$, it follows that for all $q \in \mathbb Q$ we have that $q \ge \theta$ implies $h(1) q \ge h(\theta)$ and $q \le \theta$ implies $h(1) q \le h(\theta)$. Taking the limit $q \to \theta$ in rationals, we may conclude that $h(1) \theta = h(\theta)$. Finally, for every $z = x + \theta y \in \Z + \theta \Z$ we obtain $h(z) = x h(1) + y h(\theta) = h(1) z$ as asserted.
\end{proof}

\begin{Remark}\label{rem:integers}Extending the preceding proof, the equation $h(1) \theta = h(\theta)$ is a quadratic equation with integer coef\/f\/icients. Thence for non-quadratic $\theta$ the factor $h(1)$ must be a positive integer.
\end{Remark}

Given a f\/inite group $G$ and its representation ring $R(G)$, it is a well-known fact that there is only one ring homomorphism $r\colon R(G) \to \R$ with $r(\pi) > 0$ for every $\pi \in \hat G$, namely $r(\pi) = \dim \pi$ for every $\pi \in \hat G$. The next result shows that this statement remains true in the context of f\/inite quantum groups.

\begin{Lemma}	\label{lem:charac_R(G)}	Let $\aG$ be a finite quantum group and denote by $R(\aG)$ its representation ring. Then there is only one ring homomorphism $r\colon R(\aG) \to \R$ with $r(\pi) > 0$ for every $\pi \in \hat \aG$, namely $r(\pi) = \dim \pi$ for every $\pi \in \hat \aG$.
\end{Lemma}
\begin{proof}
	Let $r_1,r_2\colon R(\aG) \to \R$ be two such positive, non-zero ring homomorphisms and let us f\/ix $\pi \in \hat \aG$. We consider the matrix $T(\pi)$ with rows and columns index by $\hat \aG$ given by
	\begin{gather*}
		T(\pi)_{\rho, \sigma} := \frac{m(\sigma, \rho \tensor \pi) r_1(\sigma)}{r_1(\rho \tensor \pi)}
	\end{gather*}
	for all $\rho,\sigma \in \hat \aG$, where $m(\sigma, \rho \tensor \pi)$ denotes the multiplicity of $\sigma$ in $\rho \tensor \pi$. A straightforward computation verif\/ies that $T(\pi)$ is a stochastic matrix. Moreover, the vector $c = (c_\rho)_{\rho \in \hat \aG}$ with $c_\rho := r_2(\rho) / r_1(\rho)$ is an eigenvector of $T(\pi)$ with eigenvalue $\lambda = r_2(\pi) / r_1(\pi)$, because the homomorphism property implies
	\begin{gather*}
		\bigl( T(\pi) c \bigr)_\rho	= \frac{1}{r_1(\rho \tensor \pi)} \sum_{\sigma \in \hat \aG} m(\sigma, \rho \tensor \pi) r_2(\sigma)
		= \frac{r_2(\rho \tensor \pi)}{r_1(\rho \tensor \pi)}	= \frac{r_2(\pi)}{r_1(\pi)} c_\rho.
	\end{gather*}
Since all eigenvalues of stochastic matrices lie in the unit disc, we now conclude that $r_2(\pi) \le r_1(\pi)$. Exchanging the role of $r_1$ and $r_2$ likewise yields $r_2(\pi) \le r_1(\pi)$ and consequently we obtain $r_1 = r_2$.
\end{proof}

\begin{Theorem}\label{thm:cleft covering}Let $\theta \in\R$ be irrational and non-quadratic. Furthermore, let $\aG$ be a finite quantum group. Then every free compact ${\rm C}^*$-dynamical system $(\aA, \aG, \alpha)$ with fixed point algebra~$\alg A_\theta^2$ is~cleft.
\end{Theorem}
\begin{proof} Let $(\aA, \aG, \alpha)$ be a free compact ${\rm C}^*$-dynamical system with $\aA^\aG = \alg A_\theta^2$ and let $(\hH, \gamma, \omega)$ be a factor system of $(\aA, \aG, \alpha)$. Then for every representation $\pi$ of $\aG$ the $^*$-homomorphism $\gamma_\pi\colon \alg A_\theta^2 \to \alg A_\theta^2 \tensor \End(\hH_\pi)$ induces a positive group homomorpism
	\begin{gather*}
		K_0(\gamma_\pi)\colon \ \Z + \theta \Z \longrightarrow \Z + \theta \Z,
	\end{gather*}
where we have identif\/ied $K_0(\alg A_\theta^2 \tensor \End(\hH_\pi))$ with $K_0(\alg A_\theta^2) = \Z + \theta \Z$. By Remark~\ref{rem:integers}, this group homomorphism must be a positive integer of the identity, say for some factor $r(\pi) > 0$. Given two representations $\pi, \rho$ of $\aG$, we clearly have $r(\pi \oplus \rho) = r(\pi) + r(\rho)$. Moreover, the coaction condition of the factor system implies that $K_0(\gamma_\rho) \circ K_0(\gamma_\pi) = K_0(\gamma_{\pi \tensor \rho})$ and therefore that $r(\rho) \cdot r(\pi) = r(\pi \tensor \rho)$. As a consequence, we may extend the map $\pi \mapsto r(\pi)$ to a ring-homomorphism $r\colon R(\aG) \to \R$. Lemma~\ref{lem:charac_R(G)} then shows that $r(\pi) = \dim(\pi)$ holds for every $\pi \in \hat \aG$ and hence we obtain
	\begin{gather*}	[\gamma_\pi(\one)]	= K_0(\gamma_\pi)[\one]	= r(\pi) \cdot [\one]= \dim(\pi) \cdot [\one]
	\end{gather*}
in $K_0(\alg A_\theta^2)$, i.e., the projections $\gamma_\pi(\one)$ and $\one_\pi \tensor \one_{\alg A_\theta^2} \in \End(V_\pi) \tensor \alg A_\theta^2$ are stably equivalent. Since stable equivalence and Murray--von~Neumann equivalence coincide for the ${\rm C}^*$-algebra $\alg A_\theta^2$ (see~\mbox{\cite{Rieffel83,Rieffel85}}), we f\/inally conclude from Lemma~\ref{lem:cleft} that $(\aA, \aG, \alpha)$ is cleft.
\end{proof}

\appendix
\section{Frames for Morita equivalence bimodules}\label{appendixA}

In this appendix we show that Morita equivalence bimodules between unital ${\rm C}^*$-algebras admit a so-called \emph{standard module frames}. Although this might be well-known to experts, we have not found such a statement explicitly discussed in the literature.

\begin{Lemma}	\label{lem:diagonalize} Let $M$ be a Morita equivalence between unital ${\rm C}^*$-algebras $\aA$ and $\aB$. Then there are elements $x_1, \dots, x_n \in M$ with
\begin{gather*} \sum\limits_{i=1}^n \lprod{\aA}{x_i, x_i} = \one.\end{gather*}
In particular, for any collection of such elements we have a Fourier decomposition given for all $x \in M$ by
\begin{gather*}
x = \sum_{i=1}^n x_k \rprod{\aB}{x_k,x}.
\end{gather*}
\end{Lemma}
\begin{proof}The linear span of left inner products $J := \lprod{\aA}{M,M}$ is a dense ideal in $\aA$. Since the invertible elements of $\aA$ form an open subset, $J$ contains invertible elements and hence $J = \aA$. That is, there are elements $x_1, \dots, x_n \in M$ and $y_1, \dots, y_n \in M$ with
\begin{gather*}
\one = \sum\limits_{i=1}^n \lprod{\aA}{x_i, y_i}.
\end{gather*} Then the Morita equivalence property implies
	\begin{gather}		\label{eq:diagonalize:ME}
		y = \one \acts y = \sum_{i=1}^n	\lprod{\aA}{x_i,y_i} \acts y = \sum_{i=1}^n x_i \acts \rprod{\aB}{y_i,y}
	\end{gather}
for every $y \in M$. Now consider the matrix $Y \in \aB \tensor M_n$ given by $Y_{i,j} := \rprod{\aB}{y_i,y_j}$ for $1 \le i,j \le n$. Since $Y$ is positive, we f\/ind a matrix $R = (R_{i,j})_{i,j}$ in $\aB \tensor M_n$ with $Y = R R^*$. Putting
\begin{gather*}
z_k := \sum\limits_{i=1}^n x_i \acts R_{i,k}
\end{gather*} for all $1 \le j \le n$ we f\/ind
	\begin{gather*}
		\sum_{k=1}^n \lprod{\aA}{z_k, z_k}= \sum_{i,j,k=1}^n \lprod{\aA}{ x_i \acts R_{i,k}, x_j \acts R_{j,k}}
		= \sum_{i,j=1}^n \lprod[\bigg]{\aA}{ x_i \acts \left( \sum_{k=1}^n R_{i,k} R_{j,k}^* \right), x_j }\\
\hphantom{\sum_{k=1}^n \lprod{\aA}{z_k, z_k}}{} = \sum_{i,j=1}^n \lprod{\aA}{ x_i \acts \rprod{\aB}{y_i,y_j}, x_j}
		\overset{\eqref{eq:diagonalize:ME}}= \sum_{j=1}^n \lprod{\aA}{y_j, x_j}	= \one.\tag*{\qed}
\end{gather*}\renewcommand{\qed}{}
\end{proof}

\subsection*{Acknowledgments}

We would like to acknowledge the Center of Excellence in Analysis and Dynamics Research (Academy of Finland, decision no.~271983 and no.~1138810) for supporting this research. The second name author also thanks the research fonds of the Department of Mathematics of the University of Hamburg. We would also like to express our greatest gratitude to the referees for providing very fruitful criticism.

\pdfbookmark[1]{References}{ref}
\LastPageEnding

\end{document}